\documentclass[12pt,reqno]{amsart}

\usepackage{graphicx,caption}
\usepackage{float}
\usepackage{hyperref}
\usepackage{fullpage}
\usepackage{pgf,tikz}
\usetikzlibrary{arrows}

\newcommand{\reg}{\operatorname{reg}}

\newcommand{\lbd}{\operatorname{lbd}}
\newcommand{\bd}{\operatorname{bd}}

\newtheorem{theorem}{Theorem}

\newtheorem{corollary}[theorem]{Corollary}

\newtheorem{example}[theorem]{Example}

\begin{document}
\title{An upper bound for the regularity of binomial edge ideals of
trees}
\author{A. V. Jayanthan}
\address{Department of Mathematics, Indian Institute of Technology
Madras, Chennai, INDIA - 600036.}
\email{jayanav@iitm.ac.in}
\author{N. Narayanan}
\address{Department of Mathematics, Indian Institute of Technology
Madras, Chennai, INDIA - 600036.}
\email{naru@iitm.ac.in}
\author{B. V. Raghavendra Rao}
\address{Department of Computer Science and Engineering, Indian
Institute of Technology Madras, Chennai, INDIA - 600036.}
\email{bvrr@iitm.ac.in}

\maketitle

\begin{abstract}
In this article we obtain an improved upper bound for the regularity
of binomial edge ideals of trees.
\end{abstract}

Let $G$ be a finite simple graph on $[n]$. The binomial edge ideal
$J_G$ is the ideal in $S=K[x_1, \ldots, x_n, y_1, \ldots, y_n]$
generated by the binomials $\{x_iy_j - x_jy_i \mid \{i, j\} \in
E(G)\}$, where $K$ is a field and $E(G)$ denotes the set of all edges
of $G$. This notion was introduced by Herzog et al., \cite{hhhkr10}
and independently by Ohtani \cite{ohtani11}. Ever since then researchers have
been trying to understand the interplay between the combinatorial
invariants of the graph $G$ and the algebraic invariants associated to
the ideal $J_G$. In particular, there have been a lot of attempts on
estimating the Castelnuovo-Mumford regularity of the binomial edge
ideals using combinatorial invariants.

It is known that $\ell \leq \reg(S/J_G) \le n-1$, where $n$ is the number of
vertices in $G$ and $\ell$ denotes the length of a longest induced
path in $G$, \cite{mm2013}. Further, in the same article, Matsuda and
Murai conjectured that $\reg(S/J_G) = n-1$ if and only if $G$ is a
path. This conjecture was settled in the affirmative by Kiani and Saeedi
Madani, \cite{km16}. 

A vertex $v$ in $G$ is said to be a \textit{cut vertex} if $G
\setminus \{v\}$ contains strictly more components than $G$.  
A block of a graph is a maximal induced subgraph without any cut
vertex and a \textit{block graph} is a graph in which every block is a
complete graph. Saeedi Madani and Kiani proved that if $c(G)$ denotes
the number of maximal cliques, then for a closed graph $G$,
$\reg(S/J_G) \leq c(G)$, \cite{mk2012}. For a block graph
$G$, $c(G)$ is same as the number of blocks in $G$. Saeedi Madani and
Kiani conjectured that the above inequality holds for all graphs. They
proved the conjecture for the case of generalized block graphs,
\cite{mk-arxiv-13}. In \cite{jnr16}, the authors obtained a lower
bound for the regularity of the binomial edge ideal of trees and
characterized the trees having minimal regularity.  Recently, Herzog
and Rinaldo computed one of the extremal Betti number of the binomial
edge ideal of a block graph and classified block graphs admitting
precisely one extremal Betti number, \cite{hr18}. As a consequence,
they generalized the lower bound obtained in \cite{jnr16} for block
graphs and also characterized the block graphs attaining the lower
bound. In \cite{mr18}, Mascia and Rinaldo computed the Krull dimension
and regularity of block graphs.

Trees are an important subclass of block graphs. For a tree $T$ on $n$
vertices, $c(T) = n-1$ thus making the bound $\reg(S/J_T) \leq c(T)$
far from being sharp. Chaudhry et al. proved that a tree $T$ is a
caterpillar tree if and only if $\reg(S/J_T) = \ell$, where $\ell$ is
the length of a longest path in $T$, \cite{cdi14}. The authors of this article
generalized this result to obtain an upper bound for the regularity of
a class of trees known as lobster trees, \cite{jnr16}. In this
article, we obtain an improved upper bound for $\reg(S/J_T)$. The
upper bound obtained is better than the presently known bound, $n-1$,
for most of the trees.

\section*{Upper bound for regularity of trees}

Let $G$ be a block graph.  If two distinct blocks in $G$ share a
vertex, then it is a cut vertex. A block is
said to be an \textit{end-block} if it contains at most one cut
vertex.  We define the block degree $\bd(v)$ of a cut vertex to be the
number of blocks incident to $v$. A spine of a block graph $G$ is
defined to be a maximum length path $P$ in $G$ where every edge of $P$
is a block in $G$. Note that it is possible that the spine is a single
vertex.  For $v \in V(G)$, let $\lbd(v)$ denote the number of large
blocks, i.e., blocks of size at least three, incident at $v$.

One of the terminology that we need is that of gluing of two graphs at
a vertex. Let $G$ be a graph. For a subset $W$ of $V(G)$, let $G[W]$
denote the induced subgraph of $G$ on the vertex set $W$. For a cut
vertex $v$ in $G$, let $G_1, \ldots, G_k$ denote the components of $G
\setminus \{v\}$. Let $G_i' = G[V(G_i) \cup \{v\}]$. Then we say that
$G$ is obtained by \textit{gluing} $G_1, \ldots, G_k$ at $v$,
\cite{rr14}.

A vertex $v$ in a graph $G$ is said to be a \textit{free vertex} if it
is part of exactly one maximal clique.  Let $G$ be a block graph and
$v$ be a vertex which is not a free vertex of a graph $G$. Let $G'$
denote the graph obtained by adding edges between all the vertices of
$N_G(v)$, $G''$ denote the graph $G \setminus \{v\}$ and $H$ denote
the graph $G' \setminus \{v\}$. Then there is an exact sequence,
\cite{ehh11, bms18}:

\begin{eqnarray}\label{ehh}
  0 \longrightarrow \frac{S}{J_G} \longrightarrow \frac{S}{J_{G'}} \oplus \frac{S}{J_{G''}}
  \longrightarrow \frac{S}{J_H} \longrightarrow 0
\end{eqnarray}

If $G$ is obtained by identifying a vertex each of $k$ cliques of size
at least three, then by \cite{jnr16}, $\reg(S/J_G) = k$. Now, we
consider block graphs having non-trivial spine.

\begin{theorem}\label{reg-bound}
Let $G$ be a connected block graph in which every block of size at
least three is an end-block. Let $P$ be a spine of $G$ of length
$\ell(G)
\geq 1$, $e_2(G) = |\{\{a,b\} \in E(G)  \setminus E(P) ~ \mid \bd(a)
\leq 2 \text{ and } \bd(b) \leq 2 \}|,
~C_G = \{v \in V(G) \setminus V(P) \mid \bd_G(v) \geq 3 \}$ and
$b(G)$ be the number of large end-blocks  that intersect the spine
$P$.  Then, \[\reg(S/J_G) \leq e_2(G) + \ell(G) + b(G)+ \sum_{v \in C_G} \max\{\lbd(v), 2\}.\]
\end{theorem}

\begin{proof} 

If there is no cut vertex in $G$, then $G$ is an edge and hence the
assertion holds, since $e_2(G) = 0, ~b(G) = 0$ and $C_G = \emptyset$.

Assume that $G$ has at least one cut vertex. Let $d(x,P)$ denote the
distance of the vertex $x$ from the spine $P$ and $\displaystyle{d(G) =
\sum_{x \text{ is a cut vertex in }G} d(x,P)}$. We
apply induction on $d(G)$. If $d(G) = 0$, then $G$ is a graph with a spine
$P$ and some cliques attached to $P$. Therefore, the assertion follows
from \cite[Theorem 4.5]{jnr16}.

Let $d(G) > 0$. Let $v$ be a cut vertex in $G$ such that $d(v,P)$ is
maximum. 
\vskip 2mm
\noindent
\textsc{Case I:} If $\bd(v) = 2$, then there exists a graph $G_1$ containing
$v$ as a free vertex and a clique $C$ such that $G$ is obtained by gluing
$G_1$ and $C$ at $v$. Then $e_2(G_1) = e_2(G) - 1$, $\ell(G) =
\ell(G_1), ~C_G = C_{G_1}$ and $b(G) = b(G_1)$. Moreover, $d(G_1) <
d(G)$. By induction, \[ \reg(S/J_{G_1}) \leq e_2(G_1) + \ell(G_1) +
b(G_1) + \sum_{v \in C_{G_1}} \max\{\lbd(v), 2\}.\]
By \cite[Theorem 3.1]{jnr16}, $\reg(S/J_G) = \reg(S/J_{G_1}) + 1$.
Hence the assertion follows.

\vskip 2mm \noindent
\textsc{Case II:} Assume that $\bd(v) \geq 3$. Then $v \in C_G$. 
Since $v$ is not a free vertex, it follows from the exact 
sequence (\ref{ehh}) that
%
%
\[\reg\left(\frac{S}{J_G}\right) \leq
  \max\left\{\reg\left(\frac{S}{J_{G'}}\oplus \frac{S}{J_{G''}}\right),
  \reg\left(\frac{S}{J_H}\right)+1\right\}.
\]
Since $H$ is an induced subgraph of $G'$, $\reg(S/J_H) \leq
\reg(S/J_{G'})$. Therefore, we get
\[\reg\left(\frac{S}{J_G}\right) \leq
  \max\left\{\reg\left(\frac{S}{J_{G''}}\right),
  \reg\left(\frac{S}{J_{G'}}\right) + 1 \right\}.\]
We show that both the entries on the right hand side of the
above inequality satisfies the bound given in the assertion.

Note that $v$ is not a cut vertex in $G'$ and if $v \neq y \in V(G)$ is a cut
vertex of $G'$, then it is a cut vertex of $G$ as well. Therefore,
$d(G') = d(G) - d(v, P) < d(G)$. We also have $C_{G'} = C_G \setminus
\{v\}$. It can be seen that $e_2(G') \leq e_2(G)$ and $\ell(G) = \ell(G')$.
By induction hypothesis, 
\[\reg(S/J_{G'}) \leq e_2(G') + \ell(G') + b(G') + 
\sum_{x \in C_{G'}} [\max\{\lbd(x), 2\}].
\]
If $d(v, P) = 1$, then $b(G') = b(G)+1$ and for every $u \in C_{G'},
~\lbd_{G'}(u) = \lbd_{G}(u)$. Therefore,
\[\sum_{x \in C_G} [\max\{\lbd_{G}(x), 2\}]
= \sum_{x \in C_{G'}} [\max\{\lbd_{G'}(x),2\}] +
\max\{\lbd_{G}(v),2\}.
\]
Hence
\begin{eqnarray*}
  \reg(S/J_{G'}) & \leq & e_2(G) + \ell(G) + b(G) + 1 +\sum_{x \in C_{G'}} 
	[\max\{\lbd_{G}(x), 2\}]  \\
	 & \leq & e_2(G) + \ell(G) + b(G)+\sum_{x \in C_G} 
	[\max\{\lbd_{G}(x), 2\}]  - 1.
\end{eqnarray*}

If $d(v,P) > 1$, then $b(G') = b(G)$. 
Further, there is a vertex $u_v \in C_{G'}$ which is the unique cut
vertex neighbor of $v$. Morever, we have $\lbd_{G'}(u_v) =
\lbd_G(u_v) + 1$ and for every $u \in C_{G'} \setminus
\{u_v\},
~\lbd_{G'}(u) = \lbd_{G}(u)$. Therefore,

\begin{eqnarray*}
  \sum_{x \in C_{G'}} [\max\{\lbd_{G}(x), 2\}]
  & = & \sum_{x \in C_{G'}\setminus \{u_v\}}
  [\max\{\lbd_{G'}(x),2\}] + \max\{\lbd_{G'}(u_v), 2\}  \\
  & = & \sum_{x \in C_{G'}\setminus \{u_v\}}
  [\max\{\lbd_{G}(x),2\}] + \max\{\lbd_{G}(u_v)+1, 2\}  \\
   & \leq & \sum_{x \in C_{G'}\setminus \{u_v\}}
  [\max\{\lbd_{G}(x),2\}] + \max\{\lbd_{G}(u_v), 2\} + 1 \\
 & \leq & \sum_{x \in C_{G'}\setminus \{u_v\}}
  [\max\{\lbd_{G}(x),2\}]  
   + \max\{\lbd_{G}(u_v), 2\} \\
   & &+ \max\{\lbd_G(v), 2\}-1 \\
  & = & \sum_{x \in C_G} [\max\{\lbd_{G}(x),2\}] - 1.
\end{eqnarray*}
Therefore, \[\reg(S/J_{G'}) \leq e_2(G) + \ell(G) + b(G) + \sum_{x \in
C_G} [\max\{\lbd_{G}(x),2\}] - 1.\]

Now we consider the graph $G'' = G \setminus \{v\}$. Let $\lbd_G(v) =
r$. Then $G''$ is the disjoint union of $G_1$ which is the
connected component of $G''$ containing $P$ and $C_1, \ldots, C_r$ 
maximal cliques on at least $2$ vertices and possibly some isolated vertices.
Hence $\reg(S/J_{G''}) =
\reg(S/J_{G_1}) + r$. For all $x \in V(G_1), ~\lbd_{G_1}(x) =
\lbd_G(x)$ and $d(G_1) = d(G) - d(v,P) < d(G)$. Therefore, by induction
hypothesis
\[
  \reg(S/J_{G_1}) \leq e_2(G'') + \ell(G_1) + b(G_1)  + \sum_{x \in
	C_{G_1}} [\max\{\lbd_{G_1}(x),2\}].
  \]
Now, there are two possibilities, namely $e_2(G'') = e_2(G) + 1$ or
$e_2(G'') = e_2(G)$.

If $e_2(G'') = e_2(G) + 1$, then the unique cut vertex neighbor $u_v$ of $v$
has block degree 2 in $G_1$. Therefore $C_{G_1} = C_{G} \setminus
\{v,u_v\}$ so that 
\begin{eqnarray*}
  \reg(S/J_{G_1}) & \leq & e_2(G) + 1 + \ell(G_1) + b(G_1)  + \sum_{x \in
	C_{G_1}} \max\{\lbd_{G_1}(x),2\}  \\
	& \leq & e_2(G) + \ell(G_1) + b(G_1) + \sum_{x \in
	  C_{G_1}} [\max\{\lbd_{G_1}(x),2\}] 
	   + \max\{\lbd_{G_1}(u_v),2\} \\
	& \leq & e_2(G) + \ell(G_1) + b(G_1) + \sum_{x \in
	  C_{G \setminus \{v\}}} [\max\{\lbd_{G}(x),2\}].
\end{eqnarray*}
Note also that $r = \lbd_G(v)$.
Hence \[\reg(S/J_{G''}) = \reg(S/J_{G_1}) + r \leq e_2(G) + \ell(G) + b(G)  + \sum_{x \in
  C_{G }} \max\{\lbd_{G}(x),2\}.\]

For the case when $e_2(G'') = e_2(G)$, we have $C_{G_1} = C_{G
\setminus \{v\}}$. Now as argued in the previous case, one can
conclude that 
\[\reg(S/J_{G''}) \leq e_2(G) + \ell(G) + b(G)  + \sum_{x \in
  C_{G }} \max\{\lbd_{G}(x),2\}.\]
\end{proof}

As an immediate consequence, we generalize \cite[Corollary 4.8]{jnr16}
to get an upper bound for the regularity of all trees.
\begin{corollary}\label{tree-ub}
Let $T$ be a tree on $[n]$ with spine $P$ of
length $\ell$. Let $e_2$ denote the number of edges that are not in
$P$ and with both end points having degree at most $2$
and $d_3$ denote the number of vertices, not in
$P$, and having degree at least $3$. Then
\[\reg(S/J_T) \leq e_2 + \ell + 2d_3. \]
\end{corollary}
\begin{proof} Following the notation of Theorem \ref{reg-bound}, 
$b(T) = 0$ and $\lbd_T(x) = 0$ for each $x \in V(T)$
so that $\max\{\lbd_T(x), 2\} = 2$. Now the assertion follows
directly from Theorem \ref{reg-bound}.
\end{proof}

\begin{example}
Here we illustrate by an example a block graph considered in Theorem
\ref{reg-bound}.

\begin{minipage}{\linewidth}
\begin{minipage}{0.65\linewidth}
Let $G$ be the graph given on the right side. Following the notation
in Theorem \ref{reg-bound}, we can see that $e_2(G) = 1, \ell(G) = 4,
b(G) = 0$ and $|C_G| = 2$. Therefore, we get $\reg(S/J_G) \leq 9$.
We have computed the regularity of this graph using Macaulay 2 and
have found that the graph attains the regularity upper bound.
\end{minipage}
  \begin{minipage}{0.3\linewidth}
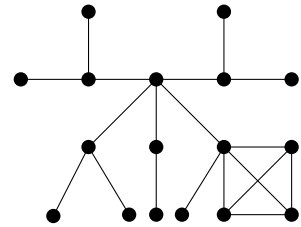
	\begin{figure}[H]
	  \begin{tikzpicture}[scale=1.8]
\draw (1.5,-0.5)-- (2,-0.5);
\draw (2,-0.5)-- (2.5,-0.5);
\draw (2.5,-0.5)-- (3,-0.5);
\draw (2.5,-0.5)-- (2.5,-1);
\draw (2,-0.5)-- (2,0);
\draw (3,-0.5)-- (3.5,-0.5);
\draw (3,-0.5)-- (3,0);
\draw (2.5,-0.5)-- (2,-1);
\draw (2,-1)-- (2.3,-1.5);
\draw (2,-1)-- (1.74,-1.51);
\draw (2.5,-0.5)-- (3,-1);
\draw (3,-1)-- (3,-1.5);
\draw (3,-1.5)-- (3.5,-1.5);
\draw (3.5,-1.5)-- (3.5,-1);
\draw (3.5,-1)-- (3,-1);
\draw (3,-1)-- (3.5,-1.5);
\draw (3.5,-1)-- (3,-1.5);
\draw (3,-1)-- (2.69,-1.5);
\draw (2.5,-1)-- (2.5,-1.5);
\begin{scriptsize}
\fill (1.5,-0.5) circle (1.5pt);
\fill (2,-0.5) circle (1.5pt);
\fill (2.5,-0.5) circle (1.5pt);
\fill (3,-0.5) circle (1.5pt);
\fill (2.5,-1) circle (1.5pt);
\fill (2,0) circle (1.5pt);
\fill (3.5,-0.5) circle (1.5pt);
\fill (3,0) circle (1.5pt);
\fill (2,-1) circle (1.5pt);
\fill (2.3,-1.5) circle (1.5pt);
\fill (1.74,-1.51) circle (1.5pt);
\fill (3,-1) circle (1.5pt);
\fill (3,-1.5) circle (1.5pt);
\fill (3.5,-1.5) circle (1.5pt);
\fill (3.5,-1) circle (1.5pt);
\fill (2.69,-1.5) circle (1.5pt);
\fill (2.5,-1.5) circle (1.5pt);
\end{scriptsize}
\end{tikzpicture}
\end{figure}
\end{minipage}
\end{minipage}

\end{example}

We also note that the upper bound we obtained in Theorem
\ref{reg-bound} coincides with the lower bound for the regularity of
\textit{Flower graph} $F_{h,k}(v)$ proved in Corollary 3.5 of
\cite{mr18}. 

\begin{corollary}
Let $F_{h,k}(v)$ denote the graph obtained by identifying a free
vertex each of $h$ copies of $C_3$ and $k\geq 1$ copies of $K_{1,3}$ at a
common vertex $v$. Then $\reg(S/J_{F_{h,k}(v)}) = 2k+h.$
\end{corollary}
\begin{proof}
Let $G = F_{h,k}(v)$. Following the notation in Theorem
\ref{reg-bound}, we get $e_2(G) = 0$ and $b(G) = h$. If $k \leq 2$, then
$C(G) = \emptyset$ and if $k > 2$, then $C(G)$ consists of all the
certer vertices of $k-2$ copies of $K_{1,3}$ outside a fixed spine.
Therefore, it follows from Theorem \ref{reg-bound} that $\reg(S/J_G)
\leq 2k + h$. Following the notation in the article \cite{mr18}, it
can be seen that $i(F(v)) = k+1$ and $cdeg(v) = h+k$. This proves the
assertion.
\end{proof}

It may also be noted that the upper bound is not attained by
all block graphs. For example, in the case of the graph considered in
\cite[Example 3.8]{mr18}, our bound gives the value $6$ while the
actual regularity is $5$.


\textbf{Acknowledgement:} We thank the Science and Engineering
Research Board (SERB) of Government of India for partially funding
this work through the Extra Mural Project Grant No. EMR/2016/001883.
We also thank the National Board for Higher Mathematics for partially
funding this work through the project No. 02011/23/2017/R\&D II/4501.
We also thank the anonymous reviewer for the valuable comments.

\bibliographystyle{abbrv}  
\bibliography{refs_reg}

\end{document}